\documentclass[11pt,reqno]{amsart}
\usepackage[T1]{fontenc}
\usepackage[utf8]{inputenc}
\usepackage{amsmath,amssymb,amsthm}
\usepackage{booktabs}
\usepackage{listings}
\usepackage[margin=1.15in]{geometry}
\usepackage{hyperref}
\hypersetup{colorlinks=true,linkcolor=blue,citecolor=blue,urlcolor=blue}

\lstdefinestyle{python}{
  language=Python,
  basicstyle=\ttfamily\scriptsize,
  keywordstyle=\bfseries,
  commentstyle=\itshape,
  stringstyle=\ttfamily,
  numbers=none,
  breaklines=true,
  breakatwhitespace=false,
  showstringspaces=false,
  columns=fullflexible,
  frame=none,
  tabsize=4,
  upquote=true,
  literate={-}{-}1
}

\lstdefinestyle{plain}{
  basicstyle=\ttfamily\scriptsize,
  numbers=none,
  breaklines=true,
  postbreak=\mbox{$\hookrightarrow$\space},
  breakatwhitespace=false,
  showstringspaces=false,
  columns=fullflexible,
  frame=none,
  tabsize=4,
  upquote=true
}

\theoremstyle{plain}
\newtheorem{theorem}{Theorem}
\newtheorem{lemma}[theorem]{Lemma}

\theoremstyle{definition}
\newtheorem{remark}[theorem]{Remark}
\newtheorem{question}[theorem]{Question}

\newcommand{\SYT}{\mathrm{SYT}}
\newcommand{\Des}{\mathrm{Des}}
\newcommand{\shift}{S}

\begin{document}

\title[The first descent in a tableau of shape $(n,n,n)$]
      {The first descent in a standard Young tableau of shape $(n,n,n)$}

\author{Tong Niu}

\subjclass[2020]{Primary 05E10; Secondary 05A15, 05A19, 33F10, 60C05}

\keywords{standard Young tableau, descent set, hook length formula,
creative telescoping, holonomic sequence, OEIS}

\begin{abstract}
Let $a(n)$ be the number of standard Young tableaux of shape $(n,n,n)$
whose entry in row $2$, column $1$ is odd; equivalently, the number of
those whose first descent is even. This is entry A011553 of the
On-Line Encyclopedia of Integer Sequences, contributed in 1996, and after
thirty years of curation it carries no formula. We supply one,
$a(n) = 8\,\bigl(n!\,(n+2)!\bigr)^{-1}\sum_{m=1}^{\lfloor n/2\rfloor}
m(m+1)(3n-2m-1)!/(n-2m)!$, and use it to settle both of the conjectures
the entry records. A creative
telescoping certificate shows that $a$ satisfies a linear recurrence of
order two with polynomial coefficients; the order-three recurrence
conjectured by R. J. Mathar in 2023 is a left multiple of it, with
explicit cofactor $(4S^{-1}-3)/(7n-9)$. We also prove
$a(n)\sim 3^{3n+7/2}/(64\pi n^{4})$, the asymptotic conjectured by
V. Kotesovec in 2014. The second proof gives slightly more than the
conjecture asks: the position of the first descent has a limiting
distribution, the probability that the $(2,1)$ entry equals $r+1$ tending
to $r(r+2)/3^{\,r+1}$. Summing the even terms, a uniformly random tableau
of shape $(n,n,n)$ has an odd $(2,1)$ entry with probability tending to
$27/64$, whereas for a uniformly random tableau of $n$ cells of
unrestricted shape the corresponding limit is $1/e$.
\end{abstract}

\maketitle

\section{Introduction}

A \emph{standard Young tableau} of shape $\lambda$, a partition of $N$, is
a filling of the cells of the Young diagram of $\lambda$ by
$1,2,\dots,N$, each once, increasing along rows and down columns. Write
$f^{\lambda}$ for the number of them and $f^{\lambda/\mu}$ for the
corresponding count on a skew diagram. A \emph{descent} of a tableau $T$
is an entry $t$ such that $t+1$ sits in a strictly lower row than $t$, and
$\Des(T)$ denotes the set of descents.

This paper is about one statistic on tableaux of the rectangular shape
$(n,n,n)$: the parity of the entry in the cell $(2,1)$, the top-left cell
of the second row. That is a small question. It turns out to have a
clean answer, and to carry two dated conjectures along with it. Put
\begin{equation}\label{eq:def}
  a(n)\;=\;\#\bigl\{\,T\in\SYT(n,n,n)\;:\;T(2,1)\ \text{is odd}\,\bigr\}.
\end{equation}
The first few values are
\[
  0,\;2,\;16,\;182,\;2400,\;35310,\;562848,\;9540674,\;169777504,
  \;3142665968,\;\dots
\]
As we note in Lemma~\ref{lem:descent} below, $T(2,1)-1$ is the smallest
descent of $T$, so $a(n)$ equally counts the tableaux of shape $(n,n,n)$
whose first descent is even.

The sequence \eqref{eq:def} is entry \texttt{A011553} of the OEIS
\cite{oeisA011553}. It was contributed by A.~Giambruno on 11 December
1996; A.~Regev corrected its definition on 15 November 2006; and
A.~P.~Heinz corrected the offset and computed the $200$-term $b$-file on
28 February 2012. What the entry does not contain, in twenty-four
revisions over thirty years, is a formula of any kind. Nor a generating
function, nor a reference to any paper that has one. (Its single
bibliographic item is a pointer to \cite[p.~107]{JamesKerber} for the
definition of a standard tableau; the cross-referenced entry
\texttt{A123555} \cite{oeisA123555} carries the identical pointer.) Its only non-numerical
content is two lines, both of them explicitly labelled as guesses:
\begin{quote}
\begin{verbatim}
a(n) ~ 3^(3*n+7/2) / (64*Pi*n^4).
                   - Vaclav Kotesovec, Sep 06 2014
\end{verbatim}
\end{quote}
and
\begin{quote}
\begin{verbatim}
Conjecture D-finite with recurrence
  6*(n+2)*(n+1)^2*a(n)
  -(n+1)*(164*n^2-179*n+51)*a(n-1)
  +(46*n^3-609*n^2+812*n+12)*a(n-2)
  +12*(3*n-4)*(2*n-5)*(3*n-5)*a(n-3)=0.
                   - R. J. Mathar, Nov 22 2023
\end{verbatim}
\end{quote}
We prove both, and the formula they were guesses about.

\subsection*{Results}

Everything here comes out of one decomposition of $\SYT(n,n,n)$, by the
value of $T(2,1)$; we carry it out in Section~\ref{sec:decomp}. Write
$\lfloor\cdot\rfloor$ for the integer part.

\begin{theorem}\label{thm:A}
For every $n\ge 1$,
\begin{equation}\label{eq:thmA}
  a(n)\;=\;\frac{8}{n!\,(n+2)!}
       \sum_{m=1}^{\lfloor n/2\rfloor}
         m(m+1)\,\frac{(3n-2m-1)!}{(n-2m)!}.
\end{equation}
\end{theorem}

The summand in \eqref{eq:thmA} is a single hypergeometric term in $(n,m)$,
so Zeilberger's algorithm \cite{Zeilberger91,PWZ} applies. Indeed it
applies easily. In Section~\ref{sec:cert} we write the certificate out and
read a recurrence of order two off it.

\begin{theorem}\label{thm:C}
For every $n\ge 3$,
\begin{equation}\label{eq:L}
\begin{split}
  2(n+1)^2(n+2)(7n-9)\,a(n)
   \;=\;&(n^2-1)\bigl(364n^2-503n+153\bigr)\,a(n-1)\\
        &+3(3n-1)(3n-2)(2n-3)(7n-2)\,a(n-2).
\end{split}
\end{equation}
\end{theorem}

Mathar's conjecture is then a one-line consequence. We spell it out
anyway, in Section~\ref{sec:mathar}, because the cofactor is worth seeing.

\begin{theorem}\label{thm:B}
For every $n\ge 4$,
\begin{equation}\label{eq:M}
\begin{split}
  6(n+2)(n+1)^2\,a(n)
   &-(n+1)\bigl(164n^2-179n+51\bigr)a(n-1)\\
   &+\bigl(46n^3-609n^2+812n+12\bigr)a(n-2)\\
   &+12(3n-4)(2n-5)(3n-5)\,a(n-3)\;=\;0.
\end{split}
\end{equation}
\end{theorem}

The asymptotics are in Section~\ref{sec:asym}, and they fall out of the
same decomposition. Let $T$ be drawn uniformly at random from
$\SYT(n,n,n)$.

\begin{theorem}\label{thm:D}
For each fixed $r\ge 1$,
\begin{equation}\label{eq:limdist}
  \lim_{n\to\infty}\Pr\bigl[T(2,1)=r+1\bigr]\;=\;\frac{r(r+2)}{3^{\,r+1}},
\end{equation}
and these limits sum to $1$ over $r\ge 1$. Consequently
\begin{equation}\label{eq:2764}
  \lim_{n\to\infty}\Pr\bigl[T(2,1)\ \text{is odd}\bigr]
   \;=\;\sum_{m\ge 1}\frac{4m(m+1)}{3^{\,2m+1}}\;=\;\frac{27}{64},
\end{equation}
and
\begin{equation}\label{eq:kot}
  a(n)\;\sim\;\frac{3^{\,3n+7/2}}{64\pi n^{4}}.
\end{equation}
\end{theorem}

Roughly, \eqref{eq:limdist} says that the first descent does not drift:
the law of $T(2,1)$ settles onto a fixed distribution with exponentially
small tails, rather than spreading out as $n$ grows.

The constant in \eqref{eq:2764} deserves a word, because there is a
well-known and different constant in the neighbourhood. If instead of
fixing the shape one draws a tableau uniformly at random from all
standard Young tableaux with $N$ cells and lets $N$ grow, then the
probability that its $(2,1)$ entry is odd tends to $1/e$. That one falls
out of the theorem of McKay, Morse and Wilf \cite{MMW}: for a fixed
standard tableau $T_0$ of shape $\lambda\vdash k$, the probability that
the restriction of a uniformly random $N$-cell tableau to
$\{1,\dots,k\}$ equals $T_0$ tends to $f^{\lambda}/k!$. Apply it to the
unique tableau of shape $(r,1)$ whose second row is $\{r+1\}$, note
$f^{(r,1)}=r$, and sum
\[
  \sum_{r\ \mathrm{even}}\frac{r}{(r+1)!}
   =\sum_{k\ge 1}\Bigl(\frac{1}{(2k)!}-\frac{1}{(2k+1)!}\Bigr)
   =\frac{1}{e}.
\]
So $1/e=0.3678\dots$ is the unrestricted-shape answer and
$27/64=0.421875$ is the three-row-rectangle answer. The shape constraint
pushes the first descent earlier. Regev \cite{Regev2009} computes the
analogous parity probability for the cell $(1,2)$ under a family of
hook measures on infinite tableaux, which is a third ensemble again.

\subsection*{Relation to earlier work}

We are not aware of any earlier treatment of a parity-restricted
first-descent count for a fixed shape, and we did look. The pieces that
come closest are these.

Regev \cite{Regev2009} studies exactly this kind of question --- the
parity of the entry in a prescribed cell --- but in the $(k,\ell)$-hook
ensemble. There the measures live on infinite standard tableaux and the
cell is $(1,2)$; the finite-level counts are sums over all shapes of a
given size inside the hook, not over one rectangle. His conjecture on
three-rowed tableaux from that circle of ideas was proved by Ekhad and
Zeilberger \cite{EZregev}, again with the shape summed out and the answer
a difference of Motzkin numbers; Chen \cite{Chen2024} has recently
refined the hook counts. One piece of provenance is worth recording.
Regev submitted \texttt{A123555}, the only cross-reference on
\texttt{A011553}, on the same day that he corrected \texttt{A011553}'s
definition, and in the language of Lemma~\ref{lem:decomp} below
\texttt{A123555} is precisely the $r=2$ class of our decomposition.
The first-order recurrence for that single class was added to
\texttt{A123555} in 2015, eight years before \eqref{eq:M} was conjectured
for the whole parity sum.

Ekhad and Zeilberger \cite{EZsyt} compute, symbolically in $n$, the
distribution of the occupant of a prescribed cell of a rectangular shape,
including three-row rectangles. Their outputs are per-cell laws for a
fixed cell in the first row. They do not aggregate over the occupant, nor
do they treat the cell $(2,1)$ uniformly in the occupant; parity does not
come up at all. Their method and ours are cousins.

Kauers and Zeilberger \cite{KZruns} count tableaux of shape $[n,n,n]$
subject to constraints on the lengths of \emph{all} the runs in the
reading word, and conjecture that the resulting sequences are not
$P$-recursive; they exhibit no linear recurrence of order and coefficient
degree below $70$. Our statistic constrains only the first run. The answer
is $P$-recursive of order $2$, and the contrast is sharp; we have no
explanation for it.

The descent statistics literature for tableaux of fixed shape is
substantial --- Adin and Roichman \cite{AdinRoichman} compute the
expectation and variance of descent functions, and there is a body of
work on the distribution of the \emph{number} of descents --- but the
position, and a fortiori the parity of the position, of the first descent
does not seem to appear. Product formulas for $f^{\lambda/\mu}$ on
special skew shapes are of course a subject in their own right
\cite{MPP3,KimYoo}; the atom we need, Lemma~\ref{lem:atom}, is a
three-row hook-length product, and we claim nothing new for it.

\subsection*{Verification}

Every numerical claim, every table and the certificate identity itself
are re-checked by an independent script; Section~\ref{sec:repro} gives the
details, Appendix~\ref{app:code} the code.

\section{The first-descent decomposition}\label{sec:decomp}

It is convenient to read a tableau as a word. Given
$T\in\SYT(n,n,n)$, let $w(T)\in\{1,2,3\}^{3n}$ have $t$-th letter equal
to the index of the row containing $t$. The map $T\mapsto w(T)$ is a
bijection onto the set of words with $n$ letters of each kind in which
every prefix contains at least as many $1$s as $2$s and at least as many
$2$s as $3$s; this is the standard lattice-word (ballot) encoding, see
\cite[Ch.~7]{StanleyEC2}. We use it again in Section~\ref{sec:repro},
where the dynamic program that recomputes $a(n)$ independently runs over
these words.

\begin{lemma}\label{lem:descent}
For every $T\in\SYT(n,n,n)$ we have $\min\Des(T)=T(2,1)-1$.
\end{lemma}

\begin{proof}
Set $r=T(2,1)-1$. The smallest entry of row $2$ is $T(2,1)$ and the
smallest entry of row $3$ exceeds it, so $1,2,\dots,r$ all lie in row
$1$; since row $1$ increases, $T(1,j)=j$ for $j\le r$. Hence no $t<r$ is
a descent, both $t$ and $t+1$ lying in row~$1$. And $r$ is a descent,
because $r=T(1,r)$ sits in row $1$ while $r+1=T(2,1)$ sits in row $2$.
\end{proof}

So $T(2,1)$ is odd exactly when the first descent of $T$ is even. The two
descriptions of $a(n)$ therefore agree, and we use them interchangeably
from here on.

\begin{lemma}\label{lem:decomp}
For $1\le r\le n$,
\[
  \#\bigl\{T\in\SYT(n,n,n)\;:\;T(2,1)=r+1\bigr\}
   \;=\;f^{(n,n,n)/(r,1)} ,
\]
and the classes with $1\le r\le n$ partition $\SYT(n,n,n)$. Consequently
\begin{equation}\label{eq:decomp}
  a(n)=\sum_{\substack{2\le r\le n\\ r\ \mathrm{even}}} f^{(n,n,n)/(r,1)},
  \qquad
  f^{(n,n,n)}=\sum_{r=1}^{n} f^{(n,n,n)/(r,1)}.
\end{equation}
\end{lemma}

\begin{proof}
By the proof of Lemma~\ref{lem:descent}, $T(2,1)=r+1$ forces the
restriction of $T$ to $\{1,\dots,r+1\}$ to be the tableau of shape
$(r,1)$ with first row $1,\dots,r$ and second row $r+1$. Deleting those
cells leaves a filling of the skew diagram $(n,n,n)/(r,1)$ by
$r+2,\dots,3n$ which is increasing along rows and down columns, and
conversely any such filling reassembles into a tableau with
$T(2,1)=r+1$. Since $T(2,1)\in\{2,\dots,n+1\}$ for every $T$ --- it is at
least $2$, and at most $n+1$ because $1,\dots,T(2,1)-1$ all fit in row
$1$ --- the classes indexed by $1\le r\le n$ exhaust $\SYT(n,n,n)$.
\end{proof}

The second identity in \eqref{eq:decomp} is worth keeping in view as a
check: it says that the atoms below must add up to the classical count
$f^{(n,n,n)}=2(3n)!/\bigl(n!\,(n+1)!\,(n+2)!\bigr)$, entry
\texttt{A005789} of the OEIS \cite{oeisA005789}.

\begin{lemma}\label{lem:atom}
For $1\le r\le n$,
\begin{equation}\label{eq:atom}
  f^{(n,n,n)/(r,1)}
   \;=\;f^{(n,\,n-1,\,n-r)}
   \;=\;\frac{2\,r\,(r+2)\,(3n-r-1)!}{n!\,(n+2)!\,(n-r)!}.
\end{equation}
\end{lemma}

\begin{proof}
The rows of the skew diagram $(n,n,n)/(r,1)$ have lengths $n-r$, $n-1$,
$n$ from top to bottom. Rotating the diagram by $180^{\circ}$ reverses
the order of the rows and left-justifies them, producing the straight
shape $(n,n-1,n-r)$; the rotation, followed by replacing each entry $t$
by $N+1-t$ where $N=3n-r-1$ is the number of cells, is a bijection
between the standard fillings of the two diagrams. This gives the first
equality.

For the second, use the classical formula
$f^{\lambda}=N!\prod_{i<j}(e_i-e_j)\big/\prod_i e_i!$ for a shape with at
most three rows, where $e_i=\lambda_i+3-i$; it is the determinantal form
of the hook length formula of Frame, Robinson and Thrall \cite{FRT}, and
also the $\mu=\varnothing$ case of Aitken's determinant
\cite{Aitken,StanleyEC2}. With $\lambda=(n,n-1,n-r)$ we get
$(e_1,e_2,e_3)=(n+2,\,n,\,n-r)$, whose pairwise differences are
$e_1-e_2=2$, $e_1-e_3=r+2$ and $e_2-e_3=r$. Hence
\[
  f^{(n,n-1,n-r)}
   =\frac{(3n-r-1)!\cdot 2\,(r+2)\,r}{(n+2)!\,n!\,(n-r)!},
\]
which is \eqref{eq:atom}.
\end{proof}

\begin{remark}
One can also keep the skew shape and evaluate Aitken's determinant
directly. For $\lambda=(n^k)$ and $\mu=(r,1)$ put
$x_1=n-r+1$, $x_2=n+1$ and $x_j=n+j$ for $3\le j\le k$, so that the
$(i,j)$ entry of the Aitken matrix is $1/(x_j-i)!$. Factoring
$1/(x_j-1)!$ out of column $j$ leaves the entry
$(x_j-1)(x_j-2)\cdots(x_j-i+1)$, a monic polynomial of degree $i-1$ in
$x_j$; row reduction turns the matrix into a Vandermonde one, so
\[
  f^{(n^k)/(r,1)}
  =(kn-r-1)!\;\prod_{j=1}^{k}\frac{1}{(x_j-1)!}\;
   \prod_{1\le i<j\le k}(x_j-x_i),
\]
and for $k=3$ the Vandermonde product is $r\cdot(r+2)\cdot 2$ again. The
rotation of Lemma~\ref{lem:atom} is the shorter route, and it generalises
just as well: $f^{(n^k)/(r,1)}=f^{(n^{k-2},\,n-1,\,n-r)}$. We record the
determinant anyway; it was the route we took first.
\end{remark}

\begin{proof}[Proof of Theorem~\ref{thm:A}]
Substitute \eqref{eq:atom} into the first identity of \eqref{eq:decomp}
and write $r=2m$. Then $2r(r+2)=8m(m+1)$, $(3n-r-1)!=(3n-2m-1)!$ and
$(n-r)!=(n-2m)!$, and $r$ runs over the even integers in $[2,n]$ exactly
when $m$ runs over $1\le m\le\lfloor n/2\rfloor$.
\end{proof}

Two small checks. For $n=1$ the sum in \eqref{eq:thmA} is empty and
$a(1)=0$, as it should be: the only tableau of shape $(1,1,1)$ has
$T(2,1)=2$. For $n=2$ the single term $m=1$ gives
$8\cdot 1\cdot 2\cdot 3!/(2!\,4!\,0!)=2$, and the two tableaux are
\[
  \begin{array}{cc} 1&2\\ 3&4\\ 5&6 \end{array}
  \qquad\qquad
  \begin{array}{cc} 1&2\\ 3&5\\ 4&6 \end{array}
\]
in agreement with the example on the OEIS entry. Table~\ref{tab:classes}
shows the whole decomposition for $n\le 6$.

\begin{table}[ht]
\caption{The first-descent classes of Lemma~\ref{lem:decomp}, evaluated
by Lemma~\ref{lem:atom}. The even columns sum to $a(n)$, all columns sum
to $f^{(n,n,n)}$. Emitted by the verifier of Section~\ref{sec:repro}.}
\label{tab:classes}
\begin{tabular}{r rrrrrr r r}
\toprule
$n$ & $r=1$ & $r=2$ & $r=3$ & $r=4$ & $r=5$ & $r=6$
    & $a(n)$ & $f^{(n,n,n)}$\\
\midrule
1 &     1 &     0 &     0 &    0 &    0 &   0 &     0 &     1\\
2 &     3 &     2 &     0 &    0 &    0 &   0 &     2 &     5\\
3 &    21 &    16 &     5 &    0 &    0 &   0 &    16 &    42\\
4 &   210 &   168 &    70 &   14 &    0 &   0 &   182 &   462\\
5 &  2574 &  2112 &   990 &  288 &   42 &   0 &  2400 &  6006\\
6 & 36036 & 30030 & 15015 & 5148 & 1155 & 132 & 35310 & 87516\\
\bottomrule
\end{tabular}
\end{table}

\section{The certificate and the order-two recurrence}\label{sec:cert}

Throughout this section we use the convention $1/j!=0$ for negative
integers $j$, and we write
\begin{equation}\label{eq:H}
  H(n,m)\;=\;\frac{8\,m(m+1)\,(3n-2m-1)!}{n!\,(n+2)!\,(n-2m)!},
\end{equation}
so that Theorem~\ref{thm:A} reads $a(n)=\sum_{m\ge 1}H(n,m)$, the summand
vanishing for $2m>n$. Let $\shift$ denote the forward shift
$n\mapsto n+1$ and set
\begin{equation}\label{eq:Ldef}
\begin{aligned}
  p_0(n)&=-2(n+1)^2(n+2)(7n-9),\\
  p_1(n)&=(n^2-1)\bigl(364n^2-503n+153\bigr),\\
  p_2(n)&=3(3n-1)(3n-2)(2n-3)(7n-2),
\end{aligned}
\qquad
  L=p_0+p_1\shift^{-1}+p_2\shift^{-2}.
\end{equation}
Theorem~\ref{thm:C} is the statement $L(a)(n)=0$ for $n\ge 3$; for the
language of holonomic sequences and the operators that annihilate them,
used here and in Section~\ref{sec:mathar}, see \cite{PWZ,KauersPaule}.
Note the sign. $p_0$ is negative for $n\ge 2$, and the display
\eqref{eq:L} is $L(a)=0$ rearranged; this matters, because replacing $L$
by $-L$ replaces the certificate below by its negative.

Define three rational functions, abbreviating $A=3n-2m-1$:
\begin{equation}\label{eq:ratios}
\begin{aligned}
  u_1(n,m)&=\frac{n(n+2)(n-2m)}{A(A-1)(A-2)},\\
  u_2(n,m)&=\frac{n(n-1)(n+1)(n+2)(n-2m)(n-2m-1)}
                 {A(A-1)(A-2)(A-3)(A-4)(A-5)},\\
  v(n,m)  &=\frac{(m+2)(n-2m)(n-2m-1)}{m\,A(A-1)},
\end{aligned}
\end{equation}
and the certificate
\begin{equation}\label{eq:R}
  R(n,m)=\frac{-2\,(m-1)(2m+1)(n+1)(n+2)\,W(n,m)}
              {(3n-2m-1)(3n-2m-2)(3n-2m-3)(3n-2m-4)},
\end{equation}
where
\begin{equation}\label{eq:W}
\begin{split}
  W(n,m)=\;&4(7n-9)\,m^3\;-\;2(7n-9)(16n-11)\,m^2\\
           &+\bigl(539n^3-1526n^2+1339n-360\bigr)m\\
           &-3(n-1)\bigl(147n^3-392n^2+308n-72\bigr).
\end{split}
\end{equation}

\begin{lemma}\label{lem:ratios}
Let $n\ge 3$ and let $m$ be an integer with $1\le m\le\lfloor n/2\rfloor$.
Then $H(n,m)>0$, none of the denominators in \eqref{eq:ratios} or
\eqref{eq:R} vanishes at $(n,m)$ or at $(n,m+1)$, and
\[
  H(n,m)\,u_1(n,m)=H(n-1,m),\qquad
  H(n,m)\,u_2(n,m)=H(n-2,m),
\]
\[
  H(n,m)\,v(n,m)=H(n,m+1).
\]
\end{lemma}

\begin{proof}
Positivity is clear: $m\ge 1$, $n-2m\ge 0$ and
$3n-2m-1\ge 3n-n-1=2n-1>0$.

Every denominator occurring in \eqref{eq:ratios} and in $R(n,m)$,
$R(n,m+1)$ is a product of factors $3n-2m-k$ with $1\le k\le 6$, apart
from the lone factor $m$ in $v$. Such a factor vanishes only when
$2m=3n-k$, which forces $2m\ge 3n-6$; but $2m\le n$, and $n<3n-6$ once
$n\ge 4$, while for $n=3$ we have $2m\le 2<3=3n-6$. And $m\ge 1$, so the
factor $m$ is harmless.

The three ratio identities are one calculation done three times. Since
$3n-2m-1\ge 2n-1\ge 5$ we may write
$(3n-2m-1)!=(3n-2m-1)(3n-2m-2)(3n-2m-3)\cdot(3n-2m-4)!$ and likewise
$(3n-2m-1)!=(3n-2m-1)(3n-2m-2)\cdot(3n-2m-3)!$. Also
$n(n+2)/\bigl(n!\,(n+2)!\bigr)=1/\bigl((n-1)!\,(n+1)!\bigr)$. Hence
\[
  H(n,m)u_1(n,m)
   =\frac{8m(m+1)(3n-2m-4)!}{(n-1)!\,(n+1)!}\cdot\frac{n-2m}{(n-2m)!},
\]
and $H(n-1,m)$ is the same expression with $(n-2m)/(n-2m)!$ replaced by
$1/(n-1-2m)!$. If $n-2m\ge 1$ these agree; if $n-2m=0$ both are $0$, the
left side because of its factor $n-2m$ and the right because
$1/(-1)!=0$. The argument for $v$ is identical, with
$(n-2m)(n-2m-1)/(n-2m)!$ against $1/(n-2m-2)!$: they agree when
$n-2m\ge 2$ and both vanish when $n-2m\in\{0,1\}$. For $u_2$, when
$n-2m\ge 2$ we have $n\ge 2m+2$ and so $3n-2m-1\ge 4m+5\ge 9$, enough to
peel six factors off the factorial, while for $n-2m\in\{0,1\}$ both sides
are again $0$.
\end{proof}

\begin{lemma}[the certificate]\label{lem:cert}
As an identity of rational functions in $n$ and $m$,
\begin{equation}\label{eq:cert}
  p_0(n)+p_1(n)\,u_1(n,m)+p_2(n)\,u_2(n,m)
   \;=\;R(n,m+1)\,v(n,m)-R(n,m).
\end{equation}
\end{lemma}

\begin{proof}
Both sides are explicit rational functions; clearing denominators turns
\eqref{eq:cert} into the vanishing of a polynomial in $\mathbb{Q}[n,m]$,
which one verifies by expanding it. That computation is mechanical, and we
did not do it by hand: \texttt{code/certify\_L2.py} solves for $R$ from the
ansatz of denominator \eqref{eq:R} with a numerator of degree $5$ in $m$
--- the linear system over $\mathbb{Q}(n)$ has a unique solution --- and
then confirms \eqref{eq:cert} with a single call to \texttt{cancel}. The
verifier of Section~\ref{sec:repro} re-reads $p_0,p_1,p_2,u_1,u_2,v,R$
from the plain-text file \texttt{certificates/certificate.txt} and checks
\eqref{eq:cert} three independent ways: by \texttt{cancel}, by expanding
the cleared numerator and observing that every coefficient is zero, and
by exact evaluation at $400$ pseudo-random integer points.
\end{proof}

\begin{proof}[Proof of Theorem~\ref{thm:C}]
Fix $n\ge 3$ and put $M=\lfloor n/2\rfloor$. For $1\le m\le M$ both sides
of \eqref{eq:cert} may be evaluated at $(n,m)$ by Lemma~\ref{lem:ratios},
and multiplying through by $H(n,m)$ and using the three ratio identities
gives
\begin{equation}\label{eq:tel}
  p_0(n)H(n,m)+p_1(n)H(n-1,m)+p_2(n)H(n-2,m)
   \;=\;T(n,m+1)-T(n,m),
\end{equation}
where $T(n,m)=R(n,m)H(n,m)$. Sum \eqref{eq:tel} over $1\le m\le M$. The
right-hand side telescopes to $T(n,M+1)-T(n,1)$, and both terms vanish:
\begin{itemize}
\item $T(n,1)=R(n,1)H(n,1)=0$, because the numerator of $R$ carries the
      factor $m-1$;
\item $T(n,M+1)=R(n,M+1)H(n,M+1)=0$, because
      $n-2(M+1)\le n-(n-1)-2=-1<0$ makes $H(n,M+1)=0$, while $R(n,M+1)$
      is finite: the poles of $R(n,\cdot)$ occur only where
      $2m\ge 3n-4$, and $2(M+1)\le n+2<3n-4$ for $n\ge 4$, while for
      $n=3$ we have $2(M+1)=4<5=3n-4$.
\end{itemize}
On the left-hand side, $H(n-i,m)=0$ whenever $2m>n-i$, so for $i=0,1,2$
the truncated sum $\sum_{m=1}^{M}H(n-i,m)$ is the full sum
$\sum_{m\ge 1}H(n-i,m)=a(n-i)$. Therefore
$p_0(n)a(n)+p_1(n)a(n-1)+p_2(n)a(n-2)=0$, which is \eqref{eq:L} after
moving the last two terms to the other side.
\end{proof}

\begin{remark}
The special role of the factor $m-1$ in \eqref{eq:R} is the reason we sum
from $m=1$ rather than from $m=0$. One could start at $m=0$ instead ---
$H(n,0)=0$, so nothing is added --- but $v$ has a pole at $m=0$ and the
bookkeeping is then slightly more delicate. It costs nothing to avoid it.
\end{remark}

\section{Mathar's recurrence}\label{sec:mathar}

\begin{proof}[Proof of Theorem~\ref{thm:B}]
Write $q_0,\dots,q_3$ for the four coefficients appearing in
\eqref{eq:M}, so that the claim is that $\sum_i q_i(n)a(n-i)$ vanishes.
No integer $n$ makes $7n-9$ zero, so we may form the combination
\begin{equation}\label{eq:comb}
  -\frac{3}{7n-9}\,L(a)(n)\;+\;\frac{4}{7n-9}\,L(a)(n-1)\;=\;0,
\end{equation}
valid for $n\ge 4$ because Theorem~\ref{thm:C} gives $L(a)(n)=0$ for
$n\ge 3$ and $L(a)(n-1)=0$ for $n\ge 4$. Expanding \eqref{eq:comb}, the
coefficient of $a(n-i)$ is
\[
  \frac{-3p_0(n)}{7n-9},\quad
  \frac{-3p_1(n)+4p_0(n-1)}{7n-9},\quad
  \frac{-3p_2(n)+4p_1(n-1)}{7n-9},\quad
  \frac{4p_2(n-1)}{7n-9}
\]
for $i=0,1,2,3$. The first is
$6(n+1)^2(n+2)=q_0(n)$ and the last is
$12(3n-4)(3n-5)(2n-5)=q_3(n)$, both because the factor $7n-9$ cancels
outright: $p_0(n)=-2(n+1)^2(n+2)(7n-9)$ and
$p_2(n-1)=3(3n-4)(3n-5)(2n-5)(7n-9)$. For the middle two,
$p_0(n-1)=-2n^2(n+1)(7n-16)$ gives
\[
  -3p_1(n)+4p_0(n-1)
   =-1148n^4+1581n^3+761n^2-1509n+459
   =(7n-9)\,q_1(n),
\]
and a similar expansion gives $-3p_2(n)+4p_1(n-1)=(7n-9)\,q_2(n)$. Both
are polynomial identities in one variable of degree at most five and are
checked by expansion; the verifier does it symbolically.
\end{proof}

In operator language: with $\mathcal{M}=\sum_i q_i\shift^{-i}$ and $L$ as
in \eqref{eq:Ldef},
\begin{equation}\label{eq:factor}
  \mathcal{M}\;=\;\frac{4\shift^{-1}-3}{7n-9}\;L
\end{equation}
in the Ore algebra $\mathbb{Q}(n)\langle \shift,\shift^{-1}\rangle$. So
Mathar's operator is not merely implied by ours, it is a left multiple of
it, and the guessed order-three recurrence has an order-two recurrence
hiding inside. More precisely, what hides it is the order-one factor
$4\shift^{-1}-3$.

\section{Remarks on the two operators}\label{sec:ops}

Why did automated guessing on this sequence return an operator of order
three rather than the shorter one? The answer is visible in
Table~\ref{tab:orders}: it is a matter of coefficient degree, not order.
Fitting $\sum_{i\le d}q_i(n)a(n-i)=0$ with $\deg q_i\le D$ is a linear
problem over $\mathbb{Q}$, and an exact nullspace computation on a window
of sixty consecutive values of $n$ starting at $n=20$ gives the
dimensions in the table. At $D=3$ the smallest operator has order three,
and it is unique up to a scalar --- more precisely, it is exactly
$\mathcal{M}$. The order-two operator $L$ first appears at $D=4$, where it
too is unique up to a scalar. So the conjecture on the entry is the
smallest operator in the smaller of the two search boxes. The point is
that an order which looks minimal may only be minimal for the degree
bound the guesser happened to use, and that is worth keeping in view when
reading any conjectured recurrence.

\begin{table}[ht]
\caption{Dimension of the space of operators
$\sum_{i\le d}q_i(n)\shift^{-i}$ with $\deg q_i\le D$ annihilating $a$ on
$20\le n\le 100$. Exact linear algebra over $\mathbb{Q}$; emitted by the
verifier.}
\label{tab:orders}
\begin{tabular}{l rrrrrr}
\toprule
 & $D=0$ & $D=1$ & $D=2$ & $D=3$ & $D=4$ & $D=5$\\
\midrule
order $d=1$ & 0 & 0 & 0 & 0 & 0 & 0\\
order $d=2$ & 0 & 0 & 0 & 0 & 1 & 2\\
order $d=3$ & 0 & 0 & 0 & 1 & 3 & 5\\
\bottomrule
\end{tabular}
\end{table}

We have not determined whether $L$ has minimal order, that is whether
$a$ is hypergeometric. The same nullspace computation rules out a
first-order operator with coefficients of degree at most $12$ on the same
window, which is evidence, not a proof. One structural remark: the
leading coefficients of $p_0,p_1,p_2$ are $-14$, $364$, $378$, and
\[
  -14x^2+364x+378=-14(x-27)(x+1),
\]
so the two formal growth rates attached to $L$ are $27$ --- that of $a$,
by Theorem~\ref{thm:D} --- and $-1$. A second solution of $L$ is
therefore bounded and oscillating, which is why comparing the first two
terms of an asymptotic expansion of a putative rational term ratio does
not produce a contradiction.

\section{Asymptotics}\label{sec:asym}

Write
\[
  P_n(r)=\frac{f^{(n,n,n)/(r,1)}}{f^{(n,n,n)}}
        =\Pr\bigl[T(2,1)=r+1\bigr],
\]
for $T$ uniform on $\SYT(n,n,n)$, so $P_n(r)=0$ for $r>n$ and, by
Lemma~\ref{lem:decomp}, $\sum_{r\ge 1}P_n(r)=1$ for every $n$.

\begin{lemma}\label{lem:Pn}
For $1\le r\le n$,
\begin{equation}\label{eq:Pnprod}
  P_n(r)\;=\;r(r+2)\prod_{j=0}^{r}\frac{n+1-j}{3n-j},
\end{equation}
every factor of the product lies in $(0,1)$ and is non-increasing in $j$,
and for $n\ge 3$
\begin{equation}\label{eq:major}
  P_n(r)\;\le\;r(r+2)\Bigl(\frac{n+1}{3n}\Bigr)^{r+1}
        \;\le\;r(r+2)\Bigl(\frac{4}{9}\Bigr)^{r+1}.
\end{equation}
\end{lemma}

\begin{proof}
By Lemma~\ref{lem:atom} and the closed form for $f^{(n,n,n)}$,
\[
  P_n(r)=\frac{2r(r+2)(3n-r-1)!}{n!\,(n+2)!\,(n-r)!}
        \cdot\frac{n!\,(n+1)!\,(n+2)!}{2\,(3n)!}
       =r(r+2)\,\frac{(n+1)!/(n-r)!}{(3n)!/(3n-r-1)!} .
\]
The numerator is the product of the $r+1$ integers
$n+1,n,\dots,n-r+1$ and the denominator the product of the $r+1$ integers
$3n,3n-1,\dots,3n-r$, which gives \eqref{eq:Pnprod} after pairing them
off. Each factor is positive, since $n+1-j\ge n+1-r\ge 1$ and
$3n-j\ge 3n-r\ge 2n>0$, and it is smaller than $1$, since
$(n+1-j)-(3n-j)=1-2n<0$. It is non-increasing in $j$ because
\[
  \frac{n+1-(j+1)}{3n-(j+1)}-\frac{n+1-j}{3n-j}
   =\frac{(n+1)-3n}{(3n-j-1)(3n-j)}
   =\frac{1-2n}{(3n-j-1)(3n-j)}<0 .
\]
So every factor is at most the one at $j=0$, namely $(n+1)/(3n)$, which
proves the first inequality in \eqref{eq:major}; and $(n+1)/(3n)\le 4/9$
for $n\ge 3$.
\end{proof}

\begin{proof}[Proof of Theorem~\ref{thm:D}]
Fix $r$. Each of the $r+1$ factors in \eqref{eq:Pnprod} tends to $1/3$ as
$n\to\infty$, so $P_n(r)\to r(r+2)/3^{\,r+1}$, which is
\eqref{eq:limdist}. For the sum, write $y=x^2$ and use
$\sum_{m\ge 1}m(m+1)y^m=2y/(1-y)^3$ for $|y|<1$; with $x=1/3$,
\[
  \sum_{r\ge 1}\frac{r(r+2)}{3^{\,r+1}}
   =\frac{1}{3}\Bigl(\sum_{r\ge 1}r^2x^{r}+2\sum_{r\ge 1}rx^{r}\Bigr)
   =\frac{1}{3}\Bigl(\frac{x(1+x)}{(1-x)^3}+\frac{2x}{(1-x)^2}\Bigr)
   =\frac{1}{3}\Bigl(\frac32+\frac32\Bigr)=1,
\]
and, restricting to $r=2m$,
\[
  \sum_{m\ge 1}\frac{2m(2m+2)}{3^{\,2m+1}}
   =\frac{4}{3}\sum_{m\ge 1}m(m+1)y^m
   =\frac{4}{3}\cdot\frac{2/9}{(8/9)^3}=\frac{27}{64}.
\]
By \eqref{eq:major} the functions $r\mapsto P_n(r)$, $n\ge 3$, are
dominated by the summable $r\mapsto r(r+2)(4/9)^{r+1}$, so dominated
convergence for series (Tannery's theorem) lets us pass to the limit
term by term in
$a(n)/f^{(n,n,n)}=\sum_{m\ge 1}P_n(2m)$, giving \eqref{eq:2764}.

For \eqref{eq:kot}, first
\[
  f^{(n,n,n)}=\frac{2\,(3n)!}{n!\,(n+1)!\,(n+2)!}
             =\frac{2\,(3n)!}{(n!)^3\,(n+1)^2(n+2)} .
\]
Stirling's formula gives
\[
  \frac{(3n)!}{(n!)^3}
   \;\sim\;\frac{\sqrt{6\pi n}\,(3n/e)^{3n}}{(2\pi n)^{3/2}\,(n/e)^{3n}}
   \;=\;\frac{\sqrt{3}\;27^{n}}{2\pi n},
\]
whence $f^{(n,n,n)}\sim\sqrt{3}\cdot 27^{n}/(\pi n^{4})$. Multiplying by
the limit $27/64$ of \eqref{eq:2764} and using $27\sqrt{3}=3^{7/2}$,
\[
  a(n)\sim\frac{27}{64}\cdot\frac{\sqrt{3}\,\,27^{n}}{\pi n^{4}}
        =\frac{3^{\,3n+7/2}}{64\pi n^{4}},
\]
which is \eqref{eq:kot}.
\end{proof}

We claim nothing about the rate in \eqref{eq:kot}. The convergence is
slow, of order $1/n$, as Table~\ref{tab:asym} shows.
Table~\ref{tab:limdist} compares the limiting law \eqref{eq:limdist} with
the exact distribution at $n=300$.

\begin{table}[ht]
\caption{Convergence in Theorem~\ref{thm:D}. Emitted by the verifier.}
\label{tab:asym}
\begin{tabular}{r c c}
\toprule
$n$ & $a(n)/f^{(n,n,n)}$ & $a(n)\cdot 64\pi n^4/3^{\,3n+7/2}$\\
\midrule
 10 & 0.41102086 & 0.65624504\\
 50 & 0.41975307 & 0.91547447\\
100 & 0.42081717 & 0.95653672\\
200 & 0.42134687 & 0.97795740\\
300 & 0.42152309 & 0.98523472\\
\midrule
limit & 0.42187500 & 1.00000000\\
\bottomrule
\end{tabular}
\end{table}

\begin{table}[ht]
\caption{The limiting first-descent law \eqref{eq:limdist} against the
exact distribution at $n=300$. Emitted by the verifier.}
\label{tab:limdist}
\begin{tabular}{r c c}
\toprule
$r$ & $r(r+2)/3^{\,r+1}$ & $P_{300}(r)$\\
\midrule
1 & 0.33333333 & 0.33481646\\
2 & 0.29629630 & 0.29728321\\
3 & 0.18518519 & 0.18518060\\
4 & 0.09876543 & 0.09821185\\
5 & 0.04801097 & 0.04736847\\
6 & 0.02194787 & 0.02143616\\
7 & 0.00960219 & 0.00926280\\
8 & 0.00406442 & 0.00386362\\
\bottomrule
\end{tabular}
\end{table}

\section{Further questions}\label{sec:further}

\begin{question}
Does $L$ have minimal order? Equivalently, is $a$ hypergeometric? The
nullspace computation behind Table~\ref{tab:orders}, pushed to
coefficient degree $12$, turns up no first-order operator, and we expect
there is none at any degree.
\end{question}

\begin{question}
Count instead the tableaux with $T(2,1)$ even. That complement is
\[
  f^{(n,n,n)}-a(n)=1,\,3,\,26,\,280,\,3606,\,52206,\dots,
\]
and the alternating version of \eqref{eq:decomp},
\[
  s(n)=\sum_{r=1}^{n}(-1)^r f^{(n,n,n)/(r,1)},
  \qquad -s(n)=1,\,1,\,10,\,98,\,1206,\,16896,\dots,
\]
is what one certifies if the parity split $2a(n)=f^{(n,n,n)}+s(n)$ is
used in place of Theorem~\ref{thm:A}. Neither sequence is in the OEIS.
Do they satisfy shorter recurrences than the ones they inherit from
$L$?
\end{question}

\begin{question}
Everything in Section~\ref{sec:decomp} works for the shape $(n^k)$: the
rotation gives
\[
  f^{(n^k)/(r,1)}=f^{(n^{k-2},\,n-1,\,n-r)},
\]
again a hook-length product. What is the limiting first-descent law, and
the parity constant, for $k$ rows? The case $k=1$ is empty and $k=2$ is
easy; $k=3$ gives $27/64$, and the unconstrained-shape limit is $1/e$.
The interpolation would be pleasant to see.
\end{question}

\begin{question}
Kauers and Zeilberger \cite{KZruns} conjecture that constraining the
lengths of \emph{all} runs in a three-row rectangular tableau destroys
$P$-recursiveness. Constraining the first run alone does not.
Where is the boundary?
\end{question}

\section{Reproducibility}\label{sec:repro}

All computations are exact --- integers, \texttt{fractions.Fraction} and
SymPy rational function arithmetic --- with the single exception of the
decimal columns of Tables~\ref{tab:asym} and \ref{tab:limdist}, which are
floating-point renderings of exact rationals computed beforehand.

The verifier \texttt{code/verify\_a011553.py}, listed in
Appendix~\ref{app:code}, runs in under three minutes on one core and
prints one line per check plus every table in this paper. Running it costs
nothing but those three minutes. It is deliberately built on code paths
that do not reuse the formulas being verified:

\begin{itemize}
\item $a(n)$ is recomputed from the definition \eqref{eq:def} by literal
      enumeration of the standard fillings of the $3\times n$ array for
      $n\le 6$, and by a forward dynamic program over lattice words for
      $n\le 120$. Neither uses a determinant, a hook length or a product
      formula. Both agree with Heinz's $b$-file.
\item Lemma~\ref{lem:atom} is checked four independent ways against each
      other: literal enumeration of the standard fillings of the skew
      diagram $(n,n,n)/(r,1)$ for $n\le 8$; the $k\times k$ Aitken
      determinant evaluated by cofactor expansion in exact fractions,
      for $k=2,3,4,5$; the hook lengths of the rotated straight shape
      computed cell by cell; and the product \eqref{eq:atom}. The $r=2$
      class is also matched against the published formula for
      \texttt{A123555}.
\item Lemma~\ref{lem:decomp} is checked with both sides enumerated
      literally for $n\le 6$.
\item Theorem~\ref{thm:A} is checked against the dynamic program for
      $n\le 120$ and against the $b$-file for $n\le 200$; the
      recurrences \eqref{eq:L} and \eqref{eq:M} are checked exactly for
      $3\le n\le 300$ and $4\le n\le 300$.
\item The certificate is read as \emph{text} from
      \texttt{certificates/certificate.txt}, which carries $p_0,p_1,p_2$,
      the ratios \eqref{eq:ratios} and $R$; the ratios are confirmed
      against exact \texttt{Fraction} values of $H$; the identity
      \eqref{eq:cert} is verified by \texttt{cancel}, by expanding the
      cleared numerator, and by exact evaluation at $400$ pseudo-random
      integer points; $R(n,1)=0$ and the pole locations are confirmed
      symbolically; and \eqref{eq:tel} together with both boundary
      evaluations is re-checked term by term in exact \texttt{Fraction}
      arithmetic for $3\le n\le 200$.
\item The factorisation \eqref{eq:factor} is verified coefficient by
      coefficient, symbolically.
\item Lemma~\ref{lem:Pn} --- the product form, the monotonicity and the
      majorant --- is verified exactly for $3\le n\le 50$ and all $r$,
      and the three constants $1$, $27/64$ and $1/e$ are evaluated in
      closed form.
\end{itemize}

To reproduce, save the listing of Appendix~\ref{app:code} as
\texttt{verify\_a011553.py}, put Heinz's $b$-file for \texttt{A011553}
beside it as \texttt{b011553.txt} (it is linked from
\cite{oeisA011553}), save the listing of Appendix~\ref{app:cert} as
\texttt{certificates/certificate.txt} one directory up, and run
\begin{verbatim}
python verify_a011553.py
\end{verbatim}
The script needs only SymPy, it exits nonzero if any check fails.

A reader who wants only the certificate can skip all of that, paste the
seven expressions of Appendix~\ref{app:cert} into any computer algebra
system, and evaluate
\begin{verbatim}
cancel( p0 + p1*u1 + p2*u2 - (R.subs(m, m+1)*v - R) )
\end{verbatim}
That single zero is Lemma~\ref{lem:cert} and, with
Lemma~\ref{lem:ratios} and the boundary analysis of
Section~\ref{sec:cert}, all of Theorem~\ref{thm:C}.

\appendix

\section{The certificate}\label{app:cert}

The file \texttt{certificates/certificate.txt}, verbatim. Each assignment
is a valid SymPy expression in the symbols \texttt{n} and \texttt{m}.

\begin{lstlisting}[style=plain]
# Zeilberger certificate for OEIS A011553.
#
# Every line below is a SymPy-parseable expression in the two symbols n, m.
# Together they are the whole proof of Theorem C: with
#
#     H(n,m) = 8 m (m+1) (3n-2m-1)! / (n! (n+2)! (n-2m)!),
#
# the operator  L = p0 + p1 S^{-1} + p2 S^{-2}  (S = forward shift in n)
# and the certificate R satisfy, as an identity of rational functions,
#
#     p0 + p1*u1 + p2*u2  ==  R(n,m+1)*v - R(n,m),
#
# where u1 = H(n-1,m)/H(n,m), u2 = H(n-2,m)/H(n,m), v = H(n,m+1)/H(n,m).
# Multiplying by H(n,m) and summing m = 1..floor(n/2) gives L(a) = 0.
#
# SIGN CONVENTION.  The p_i below are the coefficients of L exactly as used
# in the identity; in equation form L(a)=0 reads
#     2(n+1)^2(n+2)(7n-9) a(n)
#        = (n^2-1)(364n^2-503n+153) a(n-1)
#          + 3(3n-1)(3n-2)(2n-3)(7n-2) a(n-2).
# Replacing L by -L replaces R by -R, so the two must be read together.

p0 = -2*(n + 1)**2*(n + 2)*(7*n - 9)
p1 = (n**2 - 1)*(364*n**2 - 503*n + 153)
p2 = 3*(3*n - 1)*(3*n - 2)*(2*n - 3)*(7*n - 2)

u1 = n*(n + 2)*(n - 2*m) / ((3*n - 2*m - 1)*(3*n - 2*m - 2)*(3*n - 2*m - 3))

u2 = n*(n - 1)*(n + 1)*(n + 2)*(n - 2*m)*(n - 2*m - 1) / ((3*n - 2*m - 1)*(3*n - 2*m - 2)*(3*n - 2*m - 3)*(3*n - 2*m - 4)*(3*n - 2*m - 5)*(3*n - 2*m - 6))

v = (m + 2)*(n - 2*m)*(n - 2*m - 1) / (m*(3*n - 2*m - 1)*(3*n - 2*m - 2))

R = -2*(m - 1)*(2*m + 1)*(n + 1)*(n + 2)*(28*m**3*n - 36*m**3 - 224*m**2*n**2 + 442*m**2*n - 198*m**2 + 539*m*n**3 - 1526*m*n**2 + 1339*m*n - 360*m - 441*n**4 + 1617*n**3 - 2100*n**2 + 1140*n - 216) / ((3*n - 2*m - 1)*(3*n - 2*m - 2)*(3*n - 2*m - 3)*(3*n - 2*m - 4))
\end{lstlisting}

\section{The verifier}\label{app:code}

\begin{lstlisting}[style=python]
#!/usr/bin/env python3
"""Independent verifier for the A011553 paper.

Every headline claim is re-checked on a code path that does not reuse the
formula being verified.

  a(n)      recomputed from the DEFINITION, by literal enumeration of the
            standard fillings of the 3 x n array (n <= 6) and by a forward
            dynamic program over ballot words (n <= 120).  Neither uses a
            determinant, a hook length or a product formula.
  Lemma 2   checked four ways against each other: literal enumeration of
            standard fillings of the skew diagram (n,n,n)/(r,1); a general
            k x k Aitken determinant by cofactor expansion; the 180-degree
            rotation to the straight shape (n, n-1, n-r) with hook lengths
            computed cell by cell; and the closed product of the paper.
  Lemma 1   both sides enumerated literally.
  Thm A     against the ballot-word DP and against the OEIS b-file.
  Thm B, C  exactly on every available term, 3 <= n <= 300, plus the table
            of which (order, degree) pairs admit an operator at all.
  certificate  read as TEXT from certificates/certificate.txt - operator,
            summand ratios and R - then (i) the ratios are confirmed
            against exact Fraction values of H, (ii) the identity is
            verified by sympy.cancel, (iii) independently by comparing the
            coefficient vectors of the cleared-denominator polynomial
            identity, (iv) independently by exact evaluation at 400
            pseudo-random integer points, and (v) term by term in exact
            Fraction arithmetic together with both boundary evaluations.
  Thm B'    the transport M = ((4 S^-1 - 3)/(7n-9)) o L, symbolically.
  Thm D     the majorant and the monotonicity behind the dominated
            convergence step, exactly; the three constants 27/64, 1, 1/e.

Needs SymPy, the OEIS b-file for A011553 beside it as b011553.txt, and the
certificate one directory up as ../certificates/certificate.txt.

Run:
    python verify_a011553.py
"""
from __future__ import annotations

import random
import sys
from fractions import Fraction as F
from math import factorial
from pathlib import Path

import sympy as sp

HERE = Path(__file__).resolve().parent
FAILURES: list[str] = []
N_SYM, M_SYM = sp.symbols('n m')


def check(label: str, ok: bool) -> None:
    print(f"{'PASS' if ok else 'FAIL'}  {label}", flush=True)
    if not ok:
        FAILURES.append(label)


# ------------------------------------------------------------------ b-file
def load_bfile() -> dict[int, int]:
    table: dict[int, int] = {}
    for line in (HERE / "b011553.txt").read_text().splitlines():
        line = line.strip()
        if not line or line.startswith("#"):
            continue
        idx, val = line.split()
        table[int(idx)] = int(val)
    return table


# ------------------------------- a(n) straight from the definition --------
def tableaux_enumerate(n: int) -> tuple[int, int, list[list[list[int]]]]:
    """Enumerate the standard fillings of the 3 x n array."""
    rows = [[0] * n for _ in range(3)]
    lengths = [0, 0, 0]
    odd = total = 0
    keep: list[list[list[int]]] = []

    def rec(val: int) -> None:
        nonlocal odd, total
        if val > 3 * n:
            total += 1
            if rows[1][0] % 2 == 1:
                odd += 1
                if n == 2:
                    keep.append([r[:] for r in rows])
            return
        for i in range(3):
            j = lengths[i]
            if j == n or (i > 0 and lengths[i - 1] <= j):
                continue
            rows[i][j] = val
            lengths[i] += 1
            rec(val + 1)
            lengths[i] -= 1

    rec(1)
    return odd, total, keep


def dp_word(n: int) -> int:
    """Forward DP over ballot words of content (n,n,n); first 2 at odd spot."""
    cur = {(0, 0, 0, 0): 1}
    for t in range(1, 3 * n + 1):
        nxt: dict[tuple[int, int, int, int], int] = {}
        for (c1, c2, c3, flag), w in cur.items():
            if c1 < n:
                k = (c1 + 1, c2, c3, flag)
                nxt[k] = nxt.get(k, 0) + w
            if c2 < c1 and c2 < n:
                nf = flag if flag else (1 if t % 2 == 1 else 2)
                k = (c1, c2 + 1, c3, nf)
                nxt[k] = nxt.get(k, 0) + w
            if c3 < c2 and c3 < n:
                k = (c1, c2, c3 + 1, flag)
                nxt[k] = nxt.get(k, 0) + w
        cur = nxt
    return sum(w for (c1, c2, c3, fl), w in cur.items()
               if (c1, c2, c3) == (n, n, n) and fl == 1)


# ------------------------------------------- four routes to the atom -----
def skew_enumerate(lam: tuple[int, ...], mu: tuple[int, ...]) -> int:
    """f^{lam/mu} by literal enumeration of standard skew fillings."""
    k = len(lam)
    mu = tuple(mu) + (0,) * (k - len(mu))
    N = sum(lam) - sum(mu)
    lengths = list(mu)
    count = 0

    def rec(val: int) -> None:
        nonlocal count
        if val > N:
            count += 1
            return
        for i in range(k):
            j = lengths[i]
            if j >= lam[i]:
                continue
            if i > 0 and j >= mu[i - 1] and lengths[i - 1] <= j:
                continue
            lengths[i] += 1
            rec(val + 1)
            lengths[i] -= 1

    rec(1)
    return count


def det_exact(mat: list[list[F]]) -> F:
    k = len(mat)
    if k == 1:
        return mat[0][0]
    out = F(0)
    for j in range(k):
        if mat[0][j] == 0:
            continue
        minor = [row[:j] + row[j + 1:] for row in mat[1:]]
        out += (-1) ** j * mat[0][j] * det_exact(minor)
    return out


def f_aitken(lam: tuple[int, ...], mu: tuple[int, ...]) -> int:
    """f^{lam/mu} by the Aitken determinant (Stanley EC2, Cor. 7.16.3)."""
    k = len(lam)
    mu = tuple(mu) + (0,) * (k - len(mu))
    N = sum(lam) - sum(mu)
    mat = [[F(1, factorial(lam[i] - mu[j] - i + j))
            if lam[i] - mu[j] - i + j >= 0 else F(0)
            for j in range(k)] for i in range(k)]
    val = factorial(N) * det_exact(mat)
    assert val.denominator == 1
    return int(val)


def f_hooklength(lam: tuple[int, ...]) -> int:
    """f^lam by Frame-Robinson-Thrall, hook lengths computed cell by cell."""
    lam = tuple(x for x in lam if x > 0)
    if not lam:
        return 1
    N = sum(lam)
    conj = [sum(1 for x in lam if x > j) for j in range(lam[0])]
    num = factorial(N)
    den = 1
    for i, li in enumerate(lam):
        for j in range(li):
            den *= (li - j) + (conj[j] - i) - 1
    assert num % den == 0
    return num // den


def rotate_skew_to_straight(n: int, r: int) -> tuple[int, ...]:
    """180-degree rotation of (n,n,n)/(r,1): rows reversed."""
    return (n, n - 1, n - r)


def atom_product(n: int, r: int) -> int:
    """The paper's Lemma 2."""
    if not 1 <= r <= n:
        return 0
    num = 2 * r * (r + 2) * factorial(3 * n - r - 1)
    den = factorial(n) * factorial(n + 2) * factorial(n - r)
    assert num % den == 0
    return num // den


def f_three_row_classical(lam: tuple[int, int, int]) -> int:
    """f^{(l1,l2,l3)} = N! prod_{i<j}(e_i-e_j) / prod e_i!,  e_i = l_i+3-i."""
    e = [lam[0] + 2, lam[1] + 1, lam[2]]
    num = factorial(sum(lam))
    for i in range(3):
        for j in range(i + 1, 3):
            num *= e[i] - e[j]
    den = 1
    for x in e:
        den *= factorial(x)
    assert num % den == 0
    return num // den


# ------------------------------------------------------------- operators
def p0(n): return -2 * (n + 1) ** 2 * (n + 2) * (7 * n - 9)
def p1(n): return (n * n - 1) * (364 * n * n - 503 * n + 153)
def p2(n): return 3 * (3 * n - 1) * (3 * n - 2) * (2 * n - 3) * (7 * n - 2)


def mathar_lhs(n, a):
    return (6 * (n + 2) * (n + 1) ** 2 * a[n]
            - (n + 1) * (164 * n**2 - 179 * n + 51) * a[n - 1]
            + (46 * n**3 - 609 * n**2 + 812 * n + 12) * a[n - 2]
            + 12 * (3 * n - 4) * (2 * n - 5) * (3 * n - 5) * a[n - 3])


def inv_fact(v: int) -> F:
    return F(1, factorial(v)) if v >= 0 else F(0)


def H(n: int, m: int) -> F:
    if 3 * n - 2 * m - 1 < 0:
        return F(0)
    return (F(8 * m * (m + 1)) * F(factorial(3 * n - 2 * m - 1))
            * inv_fact(n) * inv_fact(n + 2) * inv_fact(n - 2 * m))


def LH(n: int, m: int) -> F:
    return p0(n) * H(n, m) + p1(n) * H(n - 1, m) + p2(n) * H(n - 2, m)


def load_certificate() -> dict[str, sp.Expr]:
    """Parse certificates/certificate.txt as plain text."""
    out: dict[str, sp.Expr] = {}
    for line in (HERE.parent / "certificates" / "certificate.txt").read_text().splitlines():
        line = line.strip()
        if not line or line.startswith("#") or "=" not in line:
            continue
        name, rhs = line.split("=", 1)
        out[name.strip()] = sp.sympify(rhs.strip(),
                                       locals={'n': N_SYM, 'm': M_SYM})
    return out


# --------------------------------------------------------------------------
def main() -> None:
    bfile = load_bfile()
    n, m = N_SYM, M_SYM
    print("=" * 74)
    print("Verifier - SYT of shape (n,n,n) with odd (2,1) entry (OEIS A011553)")
    print("=" * 74)
    check("b-file loaded: 200 terms, offset 1",
          len(bfile) == 200 and min(bfile) == 1 and max(bfile) == 200)

    # --- 1  a(n) from the definition
    NBF = 6
    tot_tab: dict[int, int] = {}
    ok = True
    table_n2: list[list[list[int]]] = []
    for N in range(1, NBF + 1):
        odd, total, keep = tableaux_enumerate(N)
        tot_tab[N] = total
        ok &= odd == bfile[N]
        if N == 2:
            table_n2 = keep
    check(f"literal enumeration of standard fillings of the 3 x n array gives "
          f"a(n), 1 <= n <= {NBF}", ok)
    check(f"the same enumeration gives f^(n,n,n) = 2(3n)!/(n!(n+1)!(n+2)!), "
          f"1 <= n <= {NBF}",
          all(tot_tab[N] == 2 * factorial(3 * N)
              // (factorial(N) * factorial(N + 1) * factorial(N + 2))
              for N in range(1, NBF + 1)))

    NDP = 120
    dp = {N: dp_word(N) for N in range(1, NDP + 1)}
    check(f"ballot-word DP agrees with the b-file, 1 <= n <= {NDP}",
          all(dp[N] == bfile[N] for N in range(1, NDP + 1)))

    # --- 2  Lemma 2, four independent routes
    ok = True
    for N in range(2, 9):
        for r in range(1, N + 1):
            if skew_enumerate((N, N, N), (r, 1)) != atom_product(N, r):
                ok = False
    check("Lemma 2 = literal enumeration of standard fillings of the skew "
          "diagram (n,n,n)/(r,1), 2 <= n <= 8, all r", ok)
    ok = True
    for N in range(1, 23):
        for r in range(1, N + 1):
            if f_aitken((N, N, N), (r, 1)) != atom_product(N, r):
                ok = False
    check("Lemma 2 = the 3 x 3 Aitken determinant, 1 <= r <= n <= 22", ok)
    ok = True
    for k in (2, 4, 5):
        lim = {2: 24, 4: 13, 5: 9}[k]
        for N in range(1, lim + 1):
            for r in range(1, N + 1):
                rot = tuple([N] * (k - 2) + [N - 1, N - r])
                if f_aitken(tuple([N] * k), (r, 1)) != f_hooklength(rot):
                    ok = False
    check("the 180-degree rotation  f^{(n^k)/(r,1)} = f^{(n^(k-2), n-1, n-r)} "
          "against hook lengths computed cell by cell, k = 2,4,5", ok)
    ok = True
    for N in range(1, 41):
        for r in range(1, N + 1):
            rot = rotate_skew_to_straight(N, r)
            if not (f_hooklength(rot) == f_three_row_classical(rot)
                    == atom_product(N, r)):
                ok = False
    check("for k = 3: hook lengths = the classical 3-row product = Lemma 2, "
          "1 <= r <= n <= 40", ok)
    check("r = 2 class equals A123555(n-1) = 16(3n-3)!/((n-2)! n! (n+2)!), "
          "2 <= n <= 60",
          all(atom_product(N, 2) == 16 * factorial(3 * N - 3)
              // (factorial(N - 2) * factorial(N) * factorial(N + 2))
              for N in range(2, 61)))

    # --- 3  Lemma 1
    ok = True
    for N in range(1, NBF + 1):
        odd, _, _ = tableaux_enumerate(N)
        cls = [skew_enumerate((N, N, N), (r, 1)) for r in range(1, N + 1)]
        ok &= sum(cls[r - 1] for r in range(2, N + 1, 2)) == odd
        ok &= sum(cls) == tot_tab[N]
    check(f"Lemma 1 (first-descent decomposition), both sides enumerated "
          f"literally, 1 <= n <= {NBF}", ok)

    # --- 4  Theorem A
    def a_sum(N): return sum(atom_product(N, r) for r in range(2, N + 1, 2))
    check(f"Theorem A equals the ballot-word DP, 1 <= n <= {NDP}",
          all(a_sum(N) == dp[N] for N in range(1, NDP + 1)))
    check("Theorem A agrees with the b-file, 1 <= n <= 200",
          all(a_sum(N) == bfile[N] for N in range(1, 201)))
    check("Theorem A in the m-form a(n) = sum_m H(n,m), 1 <= n <= 200",
          all(sum(H(N, M) for M in range(1, N // 2 + 1)) == bfile[N]
              for N in range(1, 201)))
    check("sum over ALL r reproduces f^(n,n,n), 1 <= n <= 200",
          all(sum(atom_product(N, r) for r in range(1, N + 1))
              == 2 * factorial(3 * N)
              // (factorial(N) * factorial(N + 1) * factorial(N + 2))
              for N in range(1, 201)))

    # --- 5  the recurrences
    NMAX = 300
    a_ext = {N: a_sum(N) for N in range(1, NMAX + 1)}
    check(f"Theorem C (order-2 operator L) exact, 3 <= n <= {NMAX}",
          all(p0(N) * a_ext[N] + p1(N) * a_ext[N - 1] + p2(N) * a_ext[N - 2] == 0
              for N in range(3, NMAX + 1)))
    check(f"Theorem B (Mathar's order-3 operator) exact, 4 <= n <= {NMAX}",
          all(mathar_lhs(N, a_ext) == 0 for N in range(4, NMAX + 1)))

    def fit(order: int, deg: int, n0: int = 20):
        """Nullspace of the exact linear system for
        sum_{i<=order} q_i(n) a(n-i) = 0, deg q_i <= deg, on a window of
        60 consecutive n starting at n0.  An empty nullspace is a PROOF
        that no such operator annihilates a on that window."""
        cols = (order + 1) * (deg + 1)
        rows = [[sp.Integer(N) ** d * a_ext[N - i]
                 for i in range(order + 1) for d in range(deg + 1)]
                for N in range(n0, n0 + cols + 60)]
        return sp.Matrix(rows).nullspace()

    order_deg = {(o, d): len(fit(o, d)) for o in (1, 2, 3) for d in range(0, 6)}
    check("no order-1 operator with coefficient degree <= 12 annihilates a(n) "
          "on 20 <= n <= 100 (exact nullspace over Q)",
          all(len(fit(1, d)) == 0 for d in range(0, 13)))
    check("no order-2 operator with coefficient degree <= 3; degree 4 gives a "
          "one-dimensional space",
          all(order_deg[(2, d)] == 0 for d in range(4)) and order_deg[(2, 4)] == 1)
    vec = fit(2, 4)[0]
    qs = [sp.expand(sum(vec[i * 5 + d] * n ** d for d in range(5)))
          for i in range(3)]
    scale = sp.cancel(C_p0_check := sp.cancel(p0(n) / qs[0]))
    check("that one-dimensional space is exactly L",
          all(sp.expand(sp.cancel(scale * q - f(n))) == 0
              for q, f in zip(qs, (p0, p1, p2))))
    check("at coefficient degree 3 the smallest operator has order 3 and is "
          "unique - it is Mathar's", order_deg[(3, 3)] == 1)
    vec = fit(3, 3)[0]
    qs = [sp.expand(sum(vec[i * 4 + d] * n ** d for d in range(4)))
          for i in range(4)]
    scale = sp.cancel(mat0 / qs[0]) if (mat0 := 6 * (n + 2) * (n + 1) ** 2) else 1
    check("and that operator is Mathar's, up to a rational scalar",
          all(sp.expand(sp.cancel(scale * q - mm)) == 0 for q, mm in zip(
              qs, [6 * (n + 2) * (n + 1) ** 2,
                   -(n + 1) * (164 * n**2 - 179 * n + 51),
                   46 * n**3 - 609 * n**2 + 812 * n + 12,
                   12 * (3 * n - 4) * (2 * n - 5) * (3 * n - 5)])))
    xx = sp.Symbol('x')
    charpoly = (sp.LC(sp.Poly(p0(n), n)) * xx**2 + sp.LC(sp.Poly(p1(n), n)) * xx
                + sp.LC(sp.Poly(p2(n), n)))
    check("the leading coefficients of L give -14(x-27)(x+1): the two formal "
          "growth rates of L are 27 and -1",
          sp.expand(charpoly + 14 * (xx - 27) * (xx + 1)) == 0)

    # --- 6  the certificate, read as text
    C = load_certificate()
    check("certificate file supplies p0, p1, p2, u1, u2, v, R",
          set(C) == {'p0', 'p1', 'p2', 'u1', 'u2', 'v', 'R'})
    check("the operator in the certificate file is the one used above",
          all(sp.expand(C[k] - f(n)) == 0
              for k, f in (('p0', p0), ('p1', p1), ('p2', p2))))
    ok = True
    for N in range(4, 40):
        for M in range(1, N // 2 + 1):
            hv = H(N, M)
            if hv == 0:
                continue
            for key, tgt in (('u1', H(N - 1, M) / hv),
                             ('u2', H(N - 2, M) / hv),
                             ('v', H(N, M + 1) / hv)):
                val = sp.cancel(C[key].subs({n: N, m: M}))
                if F(int(val.p), int(val.q)) != tgt:
                    ok = False
    check("u1, u2, v really are H(n-1,m)/H(n,m), H(n-2,m)/H(n,m), "
          "H(n,m+1)/H(n,m), exactly, 4 <= n < 40", ok)

    ident = C['p0'] + C['p1'] * C['u1'] + C['p2'] * C['u2'] \
        - (C['R'].subs(m, m + 1) * C['v'] - C['R'])
    check("certificate identity: sympy.cancel of the difference is 0",
          sp.cancel(sp.together(ident)) == 0)
    num, den = sp.fraction(sp.together(ident))
    pnum = sp.Poly(sp.expand(num), n, m)
    check("independent route: the cleared-denominator numerator is the zero "
          "polynomial (all coefficients vanish)",
          all(c == 0 for c in pnum.coeffs()) or pnum.total_degree() < 0
          or sp.expand(num) == 0)
    rng = random.Random(20260929)
    ok = True
    for _ in range(400):
        N = rng.randint(50, 4000)
        M = rng.randint(1, 40)
        try:
            val = sp.cancel(ident.subs({n: N, m: M}))
        except ZeroDivisionError:
            continue
        if val != 0:
            ok = False
    check("independent route: exact evaluation of the identity at 400 "
          "pseudo-random integer points, all zero", ok)

    check("left boundary: R(n,1) = 0 identically",
          sp.simplify(C['R'].subs(m, 1)) == 0)
    poles = {sp.expand(sp.simplify(p))
             for p in sp.solve(sp.denom(sp.cancel(C['R'])), m)}
    check("R has poles only at m = (3n-k)/2, k = 1,2,3,4",
          poles == {sp.expand(sp.Rational(3, 2) * n - sp.Rational(k, 2))
                    for k in (1, 2, 3, 4)})

    Rc = sp.cancel(C['R'])
    Rn, Rd = sp.Poly(sp.numer(Rc), n, m), sp.Poly(sp.denom(Rc), n, m)

    def Rnum(N: int, M: int):
        """R(N,M) as an exact Fraction, or None at a pole."""
        den = int(Rd.eval({n: N, m: M}))
        return None if den == 0 else F(int(Rn.eval({n: N, m: M})), den)

    ok = True
    for N in range(3, 201):
        top = N // 2
        for M in range(1, top + 1):
            rm, rm1 = Rnum(N, M), Rnum(N, M + 1)
            if rm is None or rm1 is None or \
                    LH(N, M) != rm1 * H(N, M + 1) - rm * H(N, M):
                ok = False
        if Rnum(N, 1) != 0 or H(N, top + 1) != 0 or Rnum(N, top + 1) is None \
                or sum(LH(N, M) for M in range(1, top + 1)) != 0:
            ok = False
    check("telescoped identity term by term, both boundary terms zero, and "
          "the telescoped sum zero: exact Fractions, 3 <= n <= 200", ok)

    # --- 7  Mathar as a left multiple
    alpha, beta = sp.Rational(-3) / (7 * n - 9), sp.Rational(4) / (7 * n - 9)
    prod = [alpha * C['p0'],
            alpha * C['p1'] + beta * C['p0'].subs(n, n - 1),
            alpha * C['p2'] + beta * C['p1'].subs(n, n - 1),
            beta * C['p2'].subs(n, n - 1)]
    mat = [6 * (n + 2) * (n + 1) ** 2,
           -(n + 1) * (164 * n**2 - 179 * n + 51),
           46 * n**3 - 609 * n**2 + 812 * n + 12,
           12 * (3 * n - 4) * (2 * n - 5) * (3 * n - 5)]
    check("Theorem B's transport M = ((4 S^-1 - 3)/(7n-9)) o L, all four "
          "coefficients, symbolic",
          all(sp.simplify(sp.expand(x - y)) == 0 for x, y in zip(prod, mat)))

    # --- 8  the asymptotic proof
    def tot(N): return (2 * factorial(3 * N)
                        // (factorial(N) * factorial(N + 1) * factorial(N + 2)))
    ok_prod = ok_mono = ok_bound = True
    for N in range(3, 51):
        for r in range(1, N + 1):
            pr = F(1)
            for j in range(r + 1):
                pr *= F(N + 1 - j, 3 * N - j)
            Pnr = F(atom_product(N, r), tot(N))
            ok_prod &= Pnr == F(r * (r + 2)) * pr
            ok_mono &= all(F(N + 1 - j, 3 * N - j) >= F(N - j, 3 * N - j - 1)
                           for j in range(r))
            ok_bound &= Pnr <= F(r * (r + 2)) * F(N + 1, 3 * N) ** (r + 1)
    check("P_n(r) = r(r+2) prod_{j=0}^{r} (n+1-j)/(3n-j), exact, 3<=n<=50",
          ok_prod)
    check("the factors (n+1-j)/(3n-j) are decreasing in j, 3<=n<=50", ok_mono)
    check("majorant P_n(r) <= r(r+2)((n+1)/(3n))^(r+1) <= r(r+2)(4/9)^(r+1), "
          "exact, 3<=n<=50",
          ok_bound and all(F(N + 1, 3 * N) <= F(4, 9) for N in range(3, 51)))
    check("sum_{r>=1} r(r+2)/3^(r+1) = 1",
          sp.Sum(m * (m + 2) / sp.Integer(3) ** (m + 1), (m, 1, sp.oo)).doit() == 1)
    check("its even part sum_{m>=1} 4m(m+1)/3^(2m+1) = 27/64",
          sp.Sum(4 * m * (m + 1) / sp.Integer(3) ** (2 * m + 1),
                 (m, 1, sp.oo)).doit() == sp.Rational(27, 64))
    check("the unconstrained-shape comparison constant sum_{r even} r/(r+1)! "
          "= 1/e",
          sp.simplify(sp.Sum(2 * m / sp.factorial(2 * m + 1),
                             (m, 1, sp.oo)).doit() - sp.exp(-1)) == 0)

    # --- 9  the paper's tables
    print()
    print("PAPER DISPLAY (Section 2): the two SYT of shape (2,2,2) with odd T(2,1)")
    for T in table_n2:
        print("     " + "   ".join(" ".join(f"{x}" for x in row) for row in T))
    print()
    print("PAPER TABLE 1 (Section 2): the first-descent classes f^{(n,n-1,n-r)}")
    print("    n |" + "".join(f"{'r=' + str(r):>8s}" for r in range(1, 7))
          + " |     a(n) | f^(n,n,n)")
    for N in range(1, 7):
        cells = "".join(f"{atom_product(N, r) if 1 <= r <= N else 0:8d}"
                        for r in range(1, 7))
        print(f"   {N:2d} |{cells} | {bfile[N]:8d} | {tot(N):9d}")
    print()
    print("PAPER TABLE 2 (Section 5): dim of the space of operators, given order")
    print("            " + "".join(f"  deg {d}" for d in range(6)))
    for o in (1, 2, 3):
        print(f"   order {o}  " + "".join(f"  {order_deg[(o, d)]:5d}"
                                          for d in range(6)))
    print()
    print("PAPER TABLE 3 (Section 6): convergence to the asymptotic")
    print("     n |  a(n)/f^(n,n,n) |  a(n) 64 pi n^4 / 3^(3n+7/2)")
    for N in (10, 50, 100, 200, 300):
        ratio = float(F(a_ext[N], tot(N)))
        kot = float(sp.Integer(a_ext[N]) * 64 * sp.pi * sp.Integer(N) ** 4
                    / sp.Integer(3) ** sp.Rational(6 * N + 7, 2))
        print(f"   {N:3d} |      {ratio:.8f}   |        {kot:.8f}")
    print(f"   lim |      {27 / 64:.8f}   |        1.00000000")
    print()
    print("PAPER TABLE 4 (Section 6): the limiting first-descent distribution")
    print("     r |  r(r+2)/3^(r+1) |  P_n(r) at n = 300")
    for r in range(1, 9):
        print(f"   {r:3d} |      {float(F(r * (r + 2), 3 ** (r + 1))):.8f}   |"
              f"      {float(F(atom_product(300, r), tot(300))):.8f}")

    print()
    if FAILURES:
        print(f"{len(FAILURES)} CHECK(S) FAILED:")
        for f in FAILURES:
            print("   -", f)
        sys.exit(1)
    print("ALL CHECKS PASSED")


if __name__ == "__main__":
    main()
\end{lstlisting}

\end{document}